\documentclass[11pt]{article}
\usepackage{latexsym,amsmath,amsthm,amssymb,amscd,amsfonts,stmaryrd,mathabx,wasysym,cite,array}
\usepackage{enumitem,linegoal}
\usepackage{enumitem}

\newtheorem{lemma}{Lemma}[section]
\newtheorem{proposition}[lemma]{Proposition}
\newtheorem{remark}[lemma]{Remark}
\newtheorem{remarks}[lemma]{Remarks}
\newtheorem{theorem}[lemma]{Theorem}
\newtheorem{definition}[lemma]{Definition}
\newtheorem{corollary}[lemma]{Corollary}

\newtheorem{cor}[lemma]{Corollary}

\newtheorem*{remark*}{Remark}

\begin{document}
\title {The Symplectic Size of a Randomly Rotated Convex Body}
\author{Efim D. Gluskin and Yaron Ostrover }
\date{}
\maketitle
\begin{abstract}
In this note we study the expected value of certain symplectic capacities of randomly rotated centrally symmetric convex bodies in the classical phase space.
\end{abstract}

\section{Introduction and Results} \label{sec:Int}

Symplectic capacities are fundamental invariants in symplectic topology which roughly speaking measure the ``symplectic size" of sets (see e.g.,~\cite{CHLS} and~\cite{Mc1} for two surveys).
The notion was originally introduced by Ekeland and Hofer in~\cite{EH1}, where a certain symplectic invariant was constructed via Hamiltonian dynamics, although the first examples of  such kind of invariants
were  constructed previously by Gromov in his pioneering work~\cite{G} using the theory of pseudo-holomorphic curves. Shortly after this, 
many other  symplectic capacities were
constructed reflecting different geometrical and dynamical properties. 
All these quantities  play an important role in symplectic topology nowadays, and are closely related with symplectic embedding obstructions on the one hand, and with the existence and behaviour of periodic orbits of Hamiltonian systems
on the other. For the definition of symplectic capacities and some discussions on their properties see e.g.,~\cite{CHLS, HZ, Mc1,Schle}.

In this note we focus on  the classical phase space ${\mathbb R}^{2n}$, equipped with the standard inner product $\langle \cdot, \cdot \rangle$ and the standard symplectic form $\omega$.
Note that  under the usual identification between ${\mathbb R}^{2n}$ and ${\mathbb C}^n$, these two structures are the real and the imaginary parts, respectively, of
the standard Hermitian inner product in ${\mathbb C}^n$. Moreover,  one has that
$\omega(v, u) = \langle v,Ju \rangle$, where the linear operator $J : {\mathbb R}^{2n} \rightarrow {\mathbb R}^{2n}$ defines the standard complex structure on 
${\mathbb C}^n$.
%
Our main interest is the study of the symplectic size of sets in the class 
of convex bodies in ${\mathbb R}^{2n}$, i.e., compact convex subsets with non-empty
interior. It turns out that even in this special class, symplectic capacities are in general very difficult to compute explicitly, and there are only
a few methods to effectively estimate them (for some exceptional cases we refer the reader e.g., to~\cite{AAKO, CCFHR,Her, Hu1,Lu,O}).

In~\cite{EH1} and~\cite{HZ1}, two symplectic capacities, nowadays known as the
Ekeland--Hofer and Hofer--Zehnder  capacities (denoted by $c_{_{\rm EH}}$ and $c_{_{\rm HZ}}$ respectively), were defined 
using a variational principle for the classical action functional from Hamiltonian dynamics.
Moreover, it was
proved (see~\cite{EH1,HZ1} and~\cite{V3})
that for a smooth convex 
body $K \subset {\mathbb R}^{2n}$, these two capacities coincide, and are given by the minimal action over
all closed characteristics on the boundary of $K$.
More preciesly, recall that if $\Sigma \subset {\mathbb R}^{2n}$ is a smooth hypersurface,  then a
closed curve $\gamma$ on $\Sigma$  is called a {\it closed characteristic} of $\Sigma$ if it is tangent to ${\rm ker}(\omega | {\Sigma})$. In other words, if $\gamma(t) + \{x \in {\mathbb R}^{2n} \, | \, \omega (\dot \gamma(t),x) =0\}$ is the tangent space to the hypersurface $\Sigma$ at $\gamma(t)$.
Recall moreover that the symplectic action of a closed curve $\gamma$ is defined by $A(\gamma) = \int_{\gamma} \lambda, $ where $\lambda = pdq$ is the Liouville 1-form, and that the action spectrum of 
${\Sigma}$ is given by
$$ {\mathcal L}({\Sigma}) = \{ |A(\gamma)| \, ; \,  \gamma \ {\rm is \ a \ closed \ characteristic \ on \ } {\Sigma} \}.$$
With these notations, the above mentioned results states that for a smooth convex body $K \subset {\mathbb R}^{2n}$ one has
%
%
%
\begin{equation} \label{EHZ-cap-def} c_{_{\rm EH}}(K) = c_{_{\rm HZ}}(K) = \min \, {\cal L}(\partial K). \end{equation}
Note that although the equalities in~$(\ref{EHZ-cap-def})$
were stated only for smooth convex bodies,
they can naturally be generalized via continuity to the class of all convex bodies 
(see e.g., Section 2.3 in~\cite{AAO}).
In the following, we shall refer to the coinciding Ekeland--Hofer
and Hofer--Zehnder capacities on this class 
 as the Ekeland--Hofer--Zehnder capacity, and denote it by $c_{_{\rm EHZ}}$.

Another important example of a symplectic capacity, which is closely related with Gromov's non-squeezing theorem~\cite{G}, is the cylindrical capacity $\overline c$.
This capacity measures the area of the base of the smallest cylinder $Z^{2n}(r) := B^2(r) \times {\mathbb C}^{n-1}$ (where $B^k(r)$ stands for the $k$-dimensional Euclidean ball of radius $r$ centered at the origin) into which a subset of ${\mathbb R}^{2n}$ (not necessarily convex) could be
symplectically embed.
An alternative description (see e.g., Appendix C in~\cite{Schle}) is   
$$\overline c(U) = \inf {\rm Area} \left (\pi_E ( \phi (U)) \right ),$$ where $\pi_E$
is the orthogonal projection to  $E = \{z \in {\mathbb C}^n \, | \, z_j =0 \ {\rm for } \ j \neq 1 \}$,
and the infimum is taken over all symplectic embeddings $\phi$ of the set $U$ into ${\mathbb R}^{2n}$. 

Recently it was proved by the authors that for centrally symmetric convex bodies  in ${\mathbb R}^{2n}$,
several symplectic capacities, including the Ekeland--Hofer--Zehnder capacity,  the cylindrical capacity, and its linearized version $\overline c_{{\rm Sp}(2n)}$ (see Definition 2.4 in~\cite{GO1}), are all
asymptotically equivalent (see Theorem 1.6 in~\cite{GO1}, and Theorem 1.1 below). In the current note we use  this fact to  estimate the
expected value of the 
Ekeland--Hofer--Zehnder capacity of a randomly rotated centrally symmetric convex body $K \subset {\mathbb R}^{2n}$, at least under some non-degeneracy assumptions.  
To state our results precisely we need first  to recall some more notations.

We equip ${\mathbb R}^n$ with the standard inner product $\langle \cdot , \cdot \rangle$, and denote by $| \cdot |$ the Euclidean norm in ${\mathbb R}^n$, and  by $S^{n-1} = \{ x \in {\mathbb R}^{n} \ | \ |x|=1 \} \subset{\mathbb R}^n$ the unit sphere. For a vector $v \in {\mathbb R}^{n}$ we denote by $\{ v \}^{\perp}$ the hyperplane orthogonal to $v$. For a centrally symmetric convex body $K$ in ${\mathbb R}^n$, i.e., a compact convex subset with non-empty interior such that $K = -K$,
the associated norm on ${\mathbb R}^{n}$ (also known as the Minkowski functional) is
defined by $\| x \|_K = \inf \{\lambda  > 0 \, | \, x \in \lambda K\}$. 
 The support function $h_K : {\mathbb R}^{n} \to {\mathbb R}$ is  $K$ defined by $h_K(u) = \sup \{ \langle x , u \rangle \, | \, x \in K\}$.
 Note that $h_K$ is a norm, 
 and that
 for a direction $u \in S^{n-1}$, the quantity $h_K(u)$ is half the width of the minimal slab orthogonal to $u$ which includes $K$.
The dual (or polar) body of $K$ is defined by
$K^{\circ} = \{ y \in {\mathbb R}^{n} \ | \ h_K(y) \leq 1 \}$.
Note that one has $h_K(u) = \| u \|_{K^{\circ}}$.
Denote by $ r(K) = \max \{r  \, : \, B^n(r) \subseteq K \}$ the inradius of $K$, i.e., the radius of
the largest ball contained in $K$, and by $R(K) = \max\{ |x| \, : \, x \in K\}$ the circumradius of $K$ i.e., the radius of the smallest ball containing $K$. The mean-width of $K$ is defined by
$$ M^*(K) = \int_{S^{n-1}} h_K(x) d\sigma_{n-1}(x),$$
where 
$\sigma_{n-1}$ is the unique rotation invariant probability measure on the unit sphere $S^{n-1}$.

For centrally-symmetric convex bodies $K_1,K_2 \subset {\mathbb R}^n$, and a linear operator $\Gamma \colon {\mathbb R}^n \to {\mathbb R}^n$, we denote by
$$\| \Gamma \|_{K_1 \rightarrow K_2} = \sup\limits_{x \in K_1} \| \Gamma x\|_{K_2} = \sup\limits_{x \in K_1} \sup\limits_{y \in K_2^{\circ}} \langle \Gamma x , y \rangle,$$  the operator norm of $\Gamma$, where the latter is considered as a map between the normed spaces $({\mathbb R}^n, \| \cdot \|_{K_1}) $ and $ ({\mathbb R}^n, \| \cdot \|_{K_2})$.
The tensor product notation $v \otimes u$ denotes the rank-one $n \times n$ matrix whose entries are $v_i u_j$, i.e.,  the matrix
corresponding to the linear operator defined by $v \otimes u(w) = \langle w,v\rangle u$.  As usual, we shall identify linear operators and their matrix representations in the standard basis,
and write $A^T$ and ${\rm Tr}(A)$ for the transpose and the trace of a matrix $A$ respectively.

In what follows, we shall use standard probabilistic notations and terminology:
 a normalized measure space $(\Omega, \nu)$ is called a probability space. A measurable function $\psi : \Omega \rightarrow {\mathbb R}$ is called a random variable, and its integral with respect to $\nu$, denoted by ${\mathbb E}_{\nu} \psi$, is referred to as the expectation of $\psi$.
We recall that the special orthogonal group ${\rm SO}(n)$ is the subgroup of the orthogonal group ${\rm O}(n)$ which consists of all orthogonal transformations
in ${\mathbb R}^n$ of determinant one. It is well known that ${\rm SO}(n)$ admits a unique Haar probability measure $\mu_n$, which is invariant
under  both left and right multiplications. When there is no doubt of confusion, we drop the subscript $n$ and write just $\mu$ to simplify the notation. Equipped with this measure, the space ${\rm SO}(n)$ becomes a probability space.

On top of the standard inner product, we equip the space ${\mathbb R}^{2n} = {\mathbb R}^n \oplus {\mathbb R}^n$ with the usual complex structure $J : {\mathbb R}^{2n} \rightarrow {\mathbb R}^{2n}$ given in coordinates by $J(x,y) = (-y,x)$. For a centrally symmetric convex body $K \subset {\mathbb R}^{2n}$ we denote
\begin{equation*} \label{def-alpha} \alpha(K) :=  \|J\|_{K^{\circ} \rightarrow K} = \sup_{x,y \in K^{\circ}}  \langle Jx,y \rangle. \end{equation*}

Finally, for two quantities $f$ and $g$, we use the notation $f \lesssim g$ as shorthand for the inequality $f \leq c g$ for some universal positive constant $c$. Whenever we write $f  \asymp g$, we mean that  $f \lesssim g$ and $f \lesssim g$.
The letters $C, C_0, c, c_0,c'$ etc.
denote positive
universal constants whose value is not necessarily the same in various
appearances.

The following was proved in~\cite{GO1}:
\begin{theorem} \label{GO-thm}
For every centrally symmetric convex body $K \subset {\mathbb R}^{2n}$
\begin{equation*} \label{eq-bound-from-GO}
 (\alpha(K))^{-1} \leq  c_{_{{\rm EHZ}}}(K) \leq  \overline c(K) \leq   \overline c_{_{{\rm Sp}(2n)}} (K) \leq 4(\alpha(K))^{-1}.
\end{equation*}
\end{theorem}

Our first result in this note concerns the expectation of the map $O \mapsto c_{_{\rm EHZ}}(OK)$, defined on the group ${\rm SO}(2n)$, where $K \subset {\mathbb R}^{2n}$ is some fixed centrally symmetric convex body.
\begin{theorem} \label{our-main-theorem-first-version}
 Let $K  \subset {\mathbb R}^{2n}$ be a centrally symmetric convex body,  and  $v \in \partial K^{\circ}$  one of the contact point of $K^{\circ}$ with its minimal circumscribed ball. Denote $L = \{v\}^{\perp} \subset {\mathbb R}^{2n}$. 
 Then, for every $0<p <1$ there exists a constant $C_p$, which depends only on $p$, such that
 \begin{equation} \label{expectations-capacity-first} C_0 {\frac {r(K)} {M^*(K^{\circ} \cap L)} } \leq  \Bigl ( {\mathbb E}_{\mu} \bigl ( \left ( {c_{_{\rm EHZ}}}(OK)  \right)^p \bigr ) \Bigr )^{1/p}
 \leq C_p {\frac {r(K)} {M^*(K^{\circ} \cap L)} },
\end{equation}
for some universal constant $C_0>0$.
 \end{theorem}

 For $p=1$, the inequality on the right-hand side of~$(\ref{expectations-capacity-first})$ does not hold for every symmetric convex body in ${\mathbb R}^{2n}$. For example, let $K_{\lambda} = B^2(1) \times B^{2n-2}(\lambda)$ for some constant $\lambda>0$. A direct computation using Theorem~\ref{GO-thm} above shows that as
 $\lambda \rightarrow \infty$, one has
 $$  {\mathbb E}_{\mu} \left({c_{_{\rm EHZ}}}(OK_{\lambda})  \right)   \rightarrow \infty, \ {\rm while} \ \ {\frac {r(K_{\lambda})} {M^*(K_{\lambda}^{\circ} \cap L)} }  \lesssim \sqrt{n}.$$
The following condition, which is motivated by the works~\cite{KV} and~\cite{LO}, is enough to extend inequality~$(\ref{expectations-capacity-first})$ for values $p \geq 1$.
\begin{definition} \label{def-non-ged-cond}
For two constants $C,q>0$, a  convex body $K \subset {\mathbb R}^n$ is said to be
``$(C,q)$-non-degenerate"  if
\begin{equation} \label{eqn-non-deg-condition}
 \int_{S^{n-1}} h_K(x) d\sigma_{n-1}(x) \leq C\left ( \int_{S^{n-1}} (h_K(x))^{-q}  d \sigma_{n-1}(x) \right )^{-{1/ q}}.
\end{equation}
\end{definition}
 \begin{theorem} \label{our-main-theorem-second-version}
  Let $K  \subset {\mathbb R}^{2n}$ be a centrally symmetric convex body,  and let $L \subset {\mathbb R}^{2n}$ as in Theorem~\ref{our-main-theorem-first-version}.
 If $K^{\circ} \cap L$ is a $(C,q)$-non-degenerate  for some $q > 0$, then 
 for every $0 < p \leq q$ 
  \begin{equation} \label{expectations-capacity11}  C_0 {\frac {r(K)} {M^*(K^{\circ} \cap L)} } \leq \Bigl ( {\mathbb E}_{\mu} \bigl ( \left ( {c_{_{\rm EHZ}}}(OK)  \right)^p \bigr ) \Bigr )^{1/p}
 \leq {4C} {\frac {r(K)} {M^*(K^{\circ} \cap L)} },
\end{equation}
where $C_0>0$
 is the same universal constant which appears in Theorem~\ref{our-main-theorem-first-version} above.
\end{theorem}
\begin{remark}
{\rm It is known (see~\cite{Gu}) that for $0<q<1$, every symmetric convex body in ${\mathbb R}^n$ is $(C,q)$-non-degenerate for some constant $C$ which depends only on $q$.
Thus, Theorem~\ref{our-main-theorem-first-version} above follows immediately from Theorem~\ref{our-main-theorem-second-version}.
}
\end{remark}

Combining a concentration of measure inequality on the special orthogonal group ${\rm SO}(2n)$ due to Gromov and Milman (Theorem~\ref{concentration-orthogonal} below), with Theorem~\ref{our-main-theorem-second-version} we obtain the following
\begin{corollary}  \label{cor-case-p=1}
For a centrally symmetric convex body $K  \subset {\mathbb  R}^{2n}$
the map $O \mapsto {c_{_{\rm EHZ}}}(OK)$ is asymptotically concentrated around its mean, i.e., there are constants $c_1,c_2>0$ such that
 \begin{equation*} \label{eqn-concentration-second}
%
  \mu \left \{ O \in {\rm SO}(2n) \, ; \, \left | {c_{_{\rm EHZ}}}(OK) -  {\mathbb E}_{\mu} \left ( {c_{_{\rm EHZ}}}(OK) \right ) \right | \geq t   \right \} \leq c_1 {\rm exp}
\left ({\frac {-c_2nt^2} {R^4(K)R^4(K^{\circ})} } \right).
 \end{equation*}
Moreover, if $L \subset {\mathbb R}^{2n}$ is the hyperplane appearing  in Theorem~\ref{our-main-theorem-first-version}, and the body $K^{\circ} \cap L$ is $(C,1)$-non-degenerate for some constant $C > 0$, then one has
\begin{equation*} \label{expectations-capacity}  {\mathbb E}_{\mu} \left ( {c_{_{\rm EHZ}}}(OK) \right )
 \asymp {\frac {r(K)} {M^*(K^{\circ} \cap L)} }.
\end{equation*}
\end{corollary}

\begin{remark}
{\rm
We remark that every centrally symmetric convex body $K \subset {\mathbb R}^n$ for which $R(K) \leq \sqrt{n}r(K)$ is $(C,1)$-non-degenerate, for some constant $C>0$.
Indeed, in this case the so called ``Dvoretzky dimension" of $K$, given by $k(K) = n (M^*(K^{\circ})/R(K^{\circ}))^2$ satisfies $k(K) \geq 1$, and the $(C,1)$-non-degeneracy condition follows from Proposition 1.2 in~\cite{KV} (cf. Corollary 1 in~\cite{LO}), and the fact that for every two centrally symmetric convex bodies $K_1,K_2 \subset {\mathbb R}^n$, if $K_2 \subset K_1 \subset \lambda K_2$ for some $\lambda >1$, and $K_2$ is $(C,q)$-non-degenerate for some $q>0$, then  $K_1$ is $(\lambda C,q)$-non-degenerate. In particular, Corollary~\ref{cor-case-p=1} above holds, for examples, for all the unit balls of the $l_p^{2n}$-norms in ${\mathbb R}^{2n}$, where $1 \leq p \leq \infty$, as well as for many other families of convex bodies.
We refer the reader to~\cite{KV} and~\cite{LO} for some other criteria that ensure inequality~$(\ref{eqn-non-deg-condition})$, and more details.
}
\end{remark}

Combined with Theorem~\ref{GO-thm} above, the main ingredient in the proof of Theorem~\ref{our-main-theorem-second-version} is the following estimate of the expectation
 of the map  
 \begin{equation*} \label{def-xi-K}  O \mapsto \alpha(OK) =  \|O^TJO\|_{K^{\circ} \rightarrow K}, \end{equation*}
defined on ${\rm SO}(2n)$.
 \begin{theorem} \label{main-theorem}
Let $K \subset {\mathbb R}^{2n}$ be a centrally-symmetric convex body, and $v \in \partial K^{\circ}$  one of the contact point of $K^{\circ}$ with its minimal circumscribed ball. Denote $L = \{v\}^{\perp} \subset {\mathbb R}^{2n}$. Then,
    \begin{equation*} \label{eqn-what-needed-to-be-proved-1} {R(K^{\circ})} {M^*(K^{\circ} \cap L)} \leq  {\mathbb E}_{\mu} \left (\alpha(OK) \right )
    \leq C_1  {R(K^{\circ})} {M^*(K^{\circ} \cap L)}, \end{equation*}
for some universal constant $C_1 >0$. Moreover, one has that 
    \begin{equation*} \label{eqn-concentration-first} \mu \left \{ O \in {\rm SO}(2n) \, ; \, \bigl   | \alpha(OK) - {\mathbb E}_{\mu} \left (\alpha(OK) \right ) \bigr | \geq t   \right \} \leq c_1 {\rm exp}      
   \left ({\frac {-c_2n t^2} {R^4(K^{\circ})}} \right)  ,\end{equation*} for some universal constants $c_1,c_2>0$.
 \end{theorem}

\noindent{\bf A Quick Proof Overview:}
For the reader's convenience, we describe briefly the main steps of the proof of Theorem~\ref{main-theorem}.
First we recall an observation proved in~\cite{GO2} which states that for a fixed unit vector $y \in S^{2n-1}$, the map $O \mapsto O^TJOy$, where $O \in {\rm SO}(2n)$, pushes forward the Haar measure on ${\rm SO}(2n)$ to the Lebesgue measure on 
 the $(2n-2)$-dimensional sphere $S^{2n-1} \cap \{y\}^{\perp}$ (see Corollary~\ref{cor-about-distribution1} below).
From this we
%
conclude that for a centrally symmetric convex body $K \subset {\mathbb R}^{2n}$, the random variable 
$O \mapsto \alpha(OK)$, defined on the group ${\rm SO}(2n)$,
is the supremum of a certain subgaussian process $\{ X_t \}_{t \in {\mathcal T}}$, defined on some metric space $({\mathcal T},d)$.  Next, a corollary of Talagrand's majorizing measure theorem is used to give an upper bound for $ {\mathbb E} \sup_{t \in {\mathcal T}} X_t$ in terms of the expected value of the supremum of a certain gaussian process $\{ Y_t \}_{t \in {\mathcal T}}$, indexed on the same set ${\mathcal T}$, and defined via the metric $d$ (see Corollary~\ref{first-proposition}).
An estimate of the latter quantity via Chevet's inequality 
completes the first part of Theorem~\ref{main-theorem}. 
The proof of the second part of the theorem is based on
 a concentration of measure inequality on the special orthogonal group due to Gromov and Milman (Theorem~\ref{concentration-orthogonal} below), combined with the fact that the map $O \mapsto \alpha(OK)$ 
has a dimension-independent Lipschitz constant.
All the above mentioned ingredients needed for the proof of Theorem~\ref{main-theorem} are presented in Section~\ref{section-preliminaries} below, and the proof itself in Section~\ref{sec-proof}.


\begin{remarks} {\rm (i) The expected values of the Ekeland--Hofer--Zehnder and cylindrical capacities for the randomly rotated cube in ${\mathbb R}^{2n}$ were computed previously 
 in~\cite{GO2}.

%


%

(ii) It is interesting to compare the ratio $r(K)/M^*(K^{\circ} \cap L)$ in Corollary~\ref{cor-case-p=1} above with some other 2-homogeneous geometric quantities associated with the body $K \subset {\mathbb R}^{2n}$. Two natural examples are the square of the inradius, and the square of the so-called volume-radius of $K$ given by $({\rm Vol}(K)/{\rm Vol}(B^{2n}))^{1/2n}$.
Table 1 above provides the asymptotic behaviour of these quantities for the following convex bodies in ${\mathbb R}^{2n}$: the standard cube $\Square^{2n}=[-1,1]^{2n}$, the croos-polytope $\Diamond^{2n}={\rm Conv}\{ \pm e_i\}$ (where $\{ e_i \}_{i=1}^{2n}$ is the standard basis of ${\mathbb R}^{2n}$),
the symplectic ellipsoid 
$${E} := {E}(a_1,\ldots,a_n)= \left \{ (z_1,\ldots,z_n) \in {\mathbb C}^n \, | \,  \sum_{i=1}^n {\frac {\pi |z_i|^2} {a_i} } < 1 \right \},$$
and the ``symplectic box" 
$$ {P} := {P}(a_1,\ldots,a_n)= \Bigl \{ (z_1,\ldots,z_n) \in {\mathbb C}^n \, | \,  0  <  \operatorname{Re}(z_i), \operatorname{Im}(z_i) < {\sqrt{a_i}} \Bigr \}.$$
In the latter two examples we assume that $0 <a_1 \leq  \cdots \leq a_n$.
The computation of the quantities appearing in Table 1 are based on standard techniques from asymptotic geometric analysis.
We remark that for any convex body $K \subset {\mathbb R}^{2n}$ and any symplectic capacity $c$, the quantity $\pi r^2(K)$ serves as a lower bound for $c(K)$, while the square of the volume-radius is known to be, up to some universal constant, an upper bound for $c(K)$ (see~\cite{AAMO}).

\begin{table}
\centering
\renewcommand{\arraystretch}{1.55}
\label{my-table1}
\begin{tabular}  { |c  |c|c|c|c|}

\hline
%
   $K$  & $r^2(K)$ & $ {\frac {r(K)} {M^*(K^{\circ} \cap L)} }$  & $ \left ( {\frac { {\rm Vol}(K)  } { {\rm Vol}(B^{2n}) } } \right )^{1/n} $  \\ \hline \hline
$\Square^{2n} $ &  1                    & $ \sqrt{ {\frac {n} {\ln(n)} } } $ & $n$   \\ \hline
 $\Diamond^{2n}$  &  ${\frac 1 n}$                 &  ${\frac 1 n}$ & ${\frac 1 n}$  \\ \hline
  ${E}$  &    $a_1$                & $ \sqrt {a_1} {\sqrt { {\frac {n } {\sum_{i=1}^n  {\tfrac 1 {a_i}} }}}}$  & $\sqrt[n]{a_1\cdots a_n}$  \\ \hline
   $ {P}$  & $a_1$                 & $\sqrt{a_1} \min\limits_{k} \sqrt{{\frac {na_k} {\ln(k)} }}$ & $n \sqrt[n]{ a_1\cdots a_n}$  \\ \hline

\end{tabular}
\caption{{\it 
Here, when we write the value of some entry to be $f(n)$, we mean that the ratio between the actual value of the entry and $f(n)$
is bounded from above and from below by two positive universal constants.}}
\end{table}

}
\end{remarks}

\noindent{{\bf Acknowledgments:}} The second-named author was partially supported by the European Research Council (ERC)
under the European Union Horizon 2020 research and innovation programme, starting grant No. 637386, and by the ISF grant No. 1274/14.

\section{Preliminaries}  \label{section-preliminaries}

In this section we recall some definitions, results, and other background material needed later on in the proofs of our main results.

\subsection{The Orlicz space $L_{\psi_2}$ and subgaussian random variables} \label{sub-sec-orlicz}

We start by recalling the definition of the Orlicz space $L_{\psi_2}$ (a more detailed discussion can be found e.g., in~\cite{RR}).
Let $\psi : [0,\infty) \rightarrow [0,\infty)$ be a convex non-decreasing function that vanishes at the origin, and let
$(\Omega,\mu)$ be a probability space. 
We denote
$$ L_{\psi} = \left \{ f: \Omega \to {\mathbb R} \ {\rm measurable}  \, \big | \, \int_{\Omega} \psi \left (   {\frac {|f(x)|} {\lambda} } \right ) d\mu(x) \leq 1, \, {\rm for \ some \ } \lambda > 0 \right \}.$$
It is a classical fact that $L_{\psi}$ is a linear space, and the functional
$\| \cdot \|_{{\psi} } : L_{\psi} \rightarrow {\mathbb R}$ given by
$$ \| f \|_{\psi} = \inf \left \{ \lambda > 0 \, | \,   \int_{\Omega} \psi \left (   {\frac {|f(x)|} {\lambda} } \right ) d\mu(x) \leq 1 \right \}$$
is a norm on  $L_{\psi}$, upon identifying functions which are equal almost everywhere as is done
with the classical $L_p$ spaces. Moreover, $(L_{\psi},\| \cdot \|_{{\psi} })$ is in fact a Banach space (see~\cite{RR}).
An important concrete example is the Orlicz space $L_{\psi_2}$ associated with the function $$\psi_2(x) = e^{x^2}-1.$$
For a probability measure space $(\Omega,\mu)$ and a random variable $Z$, one has that $Z \in L_{\psi_2}$ if and only if there is a constant $c >0$ such that
$$ {\mathbb P} \{ |Z| \geq u \} \leq 2 {\rm exp}(-u^2/c), \ {\rm for} \ {\rm all} \ u \geq 0.$$
Such a random variable $Z$ is called {\it subgaussian}. It is clear that for a subgaussian random variable $Z$ one has
$$ \|Z\|_{\psi_2} = \inf\limits_{c>0} \, {\mathbb E} \left ({\rm exp}(Z^2/c^2) \right) \leq 2. $$
Furthermore, one can check that
$$ \|Z\|_{\psi_2} \asymp \sup_{p \geq 1} {\frac 1 {\sqrt p} } \left ({\mathbb E}|Z|^p \right)^{1/p}.$$
Some classical examples of subgaussian random variables are gaussian, weighted-sum of Bernoulli's, and more general bounded random variables.
In particular, the restriction of any linear functional $f$ on ${\mathbb R}^n$ to the sphere $S^{n-1}$ is subgaussian. More preciesly, consider
$f|_{S^{n-1}}: (S^{n-1},\sigma_{n-1}) \rightarrow {\mathbb R}$, where $f(x) = \langle x,a\rangle$, and $a \in {\mathbb R}^{n}$ is some fixed vector.
In this case it is known (see e.g.,~\cite{FLM}) that 
$$\|f\|_{\psi_2}  = C_n |a|,$$ where the sequence $C_n$ satisfies $C_n \sqrt{n} \asymp 1$.


\subsection{The Distribution of $O^TAOy$ for a random $O \in {\rm SO}(2n)$ and fixed $y \in S^{2n-1}$}

Let $A \in {\mathcal L}({\mathbb R}^{2n})$ be a linear transformation of ${\mathbb R}^{2n}$,  and $y \in S^{2n-1}$ some fixed unit vector.
Denote by  $\nu^A_y$  the push-forward  measure on ${\mathbb R}^{2n}$ induced by
the Haar measure
$\mu$  on ${\rm SO}(2n)$  through the map
$$f : {\rm SO}(2n) \rightarrow {\mathbb R}^{2n}, \ {\rm defined \ by \ }
f(O) = O^TAOy.$$ 
For $v \in S^{2n-1}$, denote $ r^A_{v} =\sqrt{|Av|^2 - \langle Av,v \rangle^2},$ and let $\nu^A_{y,v}$ be the normalized Haar measure on the $(2n-2)$-dimensional sphere $S^{2n-2}(r^A_{v})$ with radius $ r^A_{v}$ which lies  in the affine hyper-space $ \langle Av,v \rangle y + \{y\}^{\perp}$.

\begin{proposition} \label{conditional-measure-proposition}
With the above notations one has\footnote[1]{This means that for any  continuous function $h \in C({\mathbb R}^{2n})$ one has $$ \int_{{\mathbb R}^{2n}} h d \nu_y^A  = \int_{v \in S^{2n-1}} \Bigl ( \int_{{\mathbb R}^{2n}} h d \nu_{y,v}^A \Bigr) \sigma_{2n-1}(v).$$ }
\begin{equation} \label{eq-about-fubini-for-measures} \nu^A_y = \int\limits_{v \in S^{2n-1}} \nu^A_{y,v} d \sigma_{2n-1}(v). \end{equation}
\end{proposition}

\begin{proof}[{\bf Proof of Proposition~\ref{conditional-measure-proposition}}]


Let $G_y = \{ U \in {\rm SO}(2n) \, | \, Uy=y \}$ be the subgroup of all the special orthogonal transformations which preserves the vector $y$. Note that one can naturally identify $G_y$ with ${\rm SO}(2n-1)$, and thus equip $G_y$ with the Haar measure $\mu_{2n-1}$. The map $O \mapsto Oy$ from ${\rm SO}(2n)$ to $S^{2n-1}$ is constant on the right $G_y$-cosets. It provides a homeomorphism between the quotient space ${\rm SO}(2n) / G_y $ and $S^{2n-1}$, and pushes forward the Haar measure on ${\rm SO}(2n)$ to that of $S^{2n-1}$. Next, for $v \in S^{2n-1}$, let $O_v \in {\rm SO}(2n)$ be some orthogonal transformation for which $O_vy=v$ (e.g., the rotation in the $\{y,v\}$-plane from $y$ to $v$). Note that the right $G_y$-coset corresponding to $v$ is $$[v] := \{ O_vU \, | \, U \in G_y \} = \{ O \in {\rm SO}(2n) \, | \, Oy=v \}.$$
It follows from the uniqueness of the Haar measure on ${\rm SO}(2n)$ that for any continuous function $\varphi \in C({\rm SO}(2n))$ one has
\begin{equation*} \label{Fubini-for-SO}
\int\limits_{{\rm SO}(2n)} \varphi(O) d\mu_{2n}(O) =  \int\limits_{S^{2n-1}} \Bigl ( \int\limits_{G_y} \varphi(O_vU) d \mu_{2n-1}(U) \Bigr ) d \sigma_{2n-1}(v).
\end{equation*}
Next, we apply the above formula for the map $\varphi = h \circ f$, where $h \in C({\mathbb R}^{2n})$ is some continuous function. By the definition of $\nu_y^A$ one has
\begin{equation} \label{Fubini-for-SO-new}
 \int\limits_{{\mathbb R}^{2n}} h(z) d \nu_y^A(z)  
= 
 \int\limits_{{\rm SO}(2n)} h(f(O))d\mu_{2n}(O) 
=  \int\limits_{ S^{2n-1}} \Bigl ( \int\limits_{G_y} h(f(O_vU)) d \mu_{2n-1}(U) \Bigr ) d \sigma_{2n-1}(v).
\end{equation}
%
To simplify the last integral we use cylindrical coordinates $(t,r,w)$ to write $z  = f(O) \in {\mathbb R}^{2n}$ as $z=ty+rw$, where $r,t \in {\mathbb R}$, $r \geq 0$, and $w \in S^{2n-1} \cap \{y \}^{\perp} \simeq S^{2n-2}$ (so that $t=\langle z,y \rangle$ and $r = \sqrt{|z|^2-t^2}$). For $v \in S^{2n-1}$, $O = O_vU \in [v]$, and $z=f(O)$ one has 
\begin{equation} \label{tzw}
 \left\{
\renewcommand*{\arraystretch}{1.4}
\begin{array}{ll}
  t(O) = \langle O^TAOy,y \rangle = \langle Av,v \rangle,\\
 r(O) = \sqrt{|Av|^2 - \langle Av,v \rangle^2},\\
 w(O) = U^Tw(O_v).
\end{array} \right.
\end{equation}
In particular, the maps $t(O) = t(v)$ and $r(O)=r(v)$ are constant on the $G_y$-right coset $[v]$. 
%
%
%
Note that for a fixed unit vector $v \in S^{2n-1}$, the point $$z = f(O_vU) = t(v)y + r(v)U^T w(O_v)$$ depends only on $U^Tw(O_v) \in S^{2n-1} \cap \{ y \}^{\perp} \simeq S^{2n-2}$. 
The map $U \mapsto U^T w(O_v)$ pushes forward the measure $\mu_{2n-1}$ on $G_y$ to the measure $\sigma_{2n-2}$ on $S^{2n-1} \cap \{ y \}^{\perp}$.  
Thus, the interior integral on the right-hand side of~$(\ref{Fubini-for-SO-new})$ equals to 
$$ \int\limits_{S^{2n-1} \cap \{y\}^{\perp}} h \bigl (t(v)y + r(v)w \bigr ) d \sigma_{2n-2}(v) = \int\limits_{{\mathbb R}^{2n}} h(z) d \nu^A_{y,v}.$$ 
Plugging this back in~$(\ref{Fubini-for-SO-new})$ one obtains that   
$$ \int\limits_{{\mathbb R}^{2n}} h d \nu_y^A  = \int\limits_{ S^{2n-1}} \Bigl ( \int_{{\mathbb R}^{2n}} h d \nu_{y,v}^A \Bigr) \sigma_{2n-1}(v),$$
and the proof of the proposition is now complete. 
\end{proof}

In the special case where $A=J$ is the linear operator associated with the standard complex structure in ${\mathbb R}^{2n} \simeq {\mathbb C}^n$, we get the following corollary obtained previously in~\cite{GO2}.
%
\begin{corollary} \label{cor-about-distribution1}
With the above notations, for $y \in S^{2n-1}$, the measure $\nu^J_y$ 
is the standard normalized rotation invariant measure on the sphere $S^{2n-1} \cap \{y \}^{\perp}$.
\end{corollary}
\begin{proof}[{\bf Proof of Corollary~\ref{cor-about-distribution1}}]
The proof follows immediately from the fact that for every vector $v \in S^{2n-1}$ one has
$$ t(v) = \langle O^TJOy,y \rangle = \langle JOy,Oy \rangle = 0, \ {\rm and} \ r(v)=1.$$ This implies that the measure  $\nu^J_{y,v}$ does not depend on $v$, and thus coincide with the unique normalized  rotation-invariant measure on $S^{2n-1} \cap \{y \}^{\perp}$.
\end{proof}

\subsection{Talagrand's comparison theorem and Chevet's inequality} \label{Sub-Sec-Talagran-Comparison}

For the purpose of this note,  a ``random process" is a just collection of (real-valued) random variables indexed by the elements of some abstract  set ${\mathcal T}$. Furthermore, a ``gaussian process"  
is a collection of centered jointly  normal random variables $\{Y_t\}_{t \in {\mathcal T}}$. 
Given a gaussian process $\{Y_t\}_{t \in {\mathcal T}}$ as above, the index set ${\mathcal T}$ can turn into a metric space by defining the distance function
\begin{equation*} \label{metric-on-index}
d(t,s):= ({\mathbb E}(Y_t-Y_s))^{1/2}, \ \ t,s \in {\mathcal T}.
\end{equation*}
The proof of the following theorem can be found in Chapter 2 of~\cite{Tal}.
\begin{theorem} [{\bf Talagrand}] \label{Talagrand-Comparison-Theorem}
Let $\{X_t\}_{t \in {\mathcal T}}$  and $\{Y_t\}_{t \in {\mathcal T}}$ be two random processes  indexed on some abstract set ${\mathcal T}$, such that for
 every $t \in {\mathcal T}$ one has ${\mathbb E}(X_t) = {\mathbb E}(Y_t) = 0$. Assume moreover that: $(i)$  $\{Y_t\}_{t \in {\mathcal T}} $ is a gaussian process, $(ii)$ the space $({\mathcal T},d)$ is a compact metric space, and $(iii)$ there is a positive constant $c_1 >0$ such that for every $t,s \in  {\mathcal T}$ one has
\begin{equation*} \label{eq-condition-in-Comparison}  \| X_t - X_s \|_{\psi_2} \leq c_1 \| Y_t - Y_s \|_{\psi_2}. \end{equation*}
Then, there is a positive constant $c_2 >0$ such that
$$ {\mathbb E} \sup_{t \in {\mathcal T}} X_t \leq c_1c_2 \, {\mathbb E} \sup_{t \in {\mathcal T}} Y_t.$$
\end{theorem}
%
%


Chevet's inequality estimates the expectation of the operator norm of a gaussian matrix (see~\cite{Ch}, c.f.~\cite{BG,Gor}). More precisely,  

\begin{theorem}  [{\bf Chevet's inequality}]  \label{Chevet's-Inequality}
Let $K_1, K_2 \subset {\mathbb R}^{n}$ be symmetric convex bodies, and $G$ an $n \times n$ matrix whose entries are  standard  i.i.d. ${\mathcal N}(0,1)$ gaussian variables.
Then,
\begin{equation*}\label{Chevet-ineq-eq}
{\mathbb E} \| G \|_{K_1 \rightarrow K_2} \leq c  \sqrt{n} \Bigl (   R(K_1) M^*(K_2^{\circ}) + R(K_2^{\circ})M^*(K_1) \Bigr ).
\end{equation*}
for some absolute constant $c>0$.
\end{theorem}

\begin{remark}
{\rm
We remark that Theorem~\ref{Chevet's-Inequality} is often formulated in the literature using the gaussian mean-width instead of the spherical mean-width. However, these two quantities are known to be asymptotically equivalent up to a factor of $\sqrt{n}$. 
}
\end{remark}

\subsection{Concentration of measure on the special orthogonal group} \label{sec-concentration}

Here we recall the  concentration
of measure inequality on the special orthogonal group obtained by Gromov and Milman in~\cite{GM}.
The group ${\rm SO}(n)$ admits
a natural Riemannian metric $d$,
which it inherits from the obvious embedding into ${\mathbb R}^{n^2}$. It is well known that the geodesic distance $d$ is asymptotically equivalent to the Hilbert--Schmidt distance i.e.,
$$  \|O_1 - O_2 \|_{2} \leq d(O_1,O_2) \leq {\frac {\pi} 2} \|O_1 - O_2 \|_{2},  \ {\rm for \ any }  \  O_1,O_2 \in {\rm SO}(n),$$
where $\| \cdot  \|_{2}$ is the   Hilbert--Schmidt norm, i.e., $\| A \|_{2} = \sqrt{\sum_{i,j=1}^n|a_{i,j}|^2}$, for an $n \times n$ matrix $A = (a_{i,j})$.
With the above notations 
one has the following inequality (see~\cite{GM,Led}):
\begin{theorem}  [{\bf Gromov--Milman}]  \label{concentration-orthogonal}
Let $n \geq 1$, $\varepsilon >0$, and $f : {\rm SO}(n) \rightarrow {\mathbb R}$ such that there exist a constant $L>0$ with
$$ f(O_1) - f(O_2) \leq L \| O_1-O_2\|_2, \ {\rm for \ all \ } O_1,O_2 \in {\rm SO}(n).$$
Then,
$$ \mu \{ O \in {\rm SO}(n) \, : \, |f(O) - {\mathbb E}_{\mu}(f) | \geq t \} \leq C{\rm exp}(-cnt^2/L^2),$$
for some universal constants $c,C>0$.
\end{theorem}

\section{Proof of the Main Results} \label{sec-proof}
In this section we prove Theorem~\ref{our-main-theorem-second-version}, Corollary~\ref{cor-case-p=1}, and Theorem~\ref{main-theorem}. 
We start with some preparation.
First, for notation convenience, we shall use the following abbreviation: $J(O) = O^TJO$, where $O \in {\rm SO}(2n)$, and $J$ is the standard linear complex structure in ${\mathbb R}^{2n}$.
For a linear operator $S \in {\mathcal L}({\mathbb R}^{2n})$, we define the random variable $\xi_S \colon {\rm SO}(2n) \rightarrow {\mathbb R}$ by
\begin{equation} \label{def-of-xi-s} \xi_S(O) = {\rm Tr}(J(O)S), \  {\rm for} \ O \in {\rm SO}(2n).
\end{equation}
Moreover, for a pair $(x,y) \in {\mathbb R}^{2n} \times {\mathbb R}^{2n}$, we define the random variable $\xi_{x,y} \colon {\rm SO}(2n) \rightarrow {\mathbb R}$ by
\begin{equation} \label{def-of-xi-x-y} \xi_{x,y}(O) = \langle J(O)x,y \rangle, \  {\rm for} \ O \in {\rm SO}(2n). \end{equation}
Next, recall that the Schatten $p$-norm ($p \geq 1$) of a linear operator  $A \in {\mathcal L}({\mathbb R}^{2n})$  is given by
$$\mathfrak \|A\|_p := \left ( \sum_{i=1}^{2n} s_k^p(A) \right )^{1/p},$$
where $s_1(A) \geq s_2(A) \geq \cdots \geq s_{2n}(A) \geq 0$ are the singular values of $A$, i.e., the eigenvalues of the Hermitian operator $\sqrt{A^TA}$.
Two notable cases, which will be used in the sequel, are the trace-class norm 
$\| A \|_1$, 
and the Hilbert--Schmidt norm $\| A \|_2$, which was defined in an equivalent from in Section~\ref{sec-concentration} above. 
\begin{lemma} \label{lemma-about-Schatten}
There exists a positive constant $c>0$ such that
\begin{enumerate}
\item For any pair $(x,y) \in {\mathbb R}^{2n} \times {\mathbb R}^{2n}$  one has  $\|\xi_{x,y}\|_{\psi_2} \leq {\frac {c} {\sqrt n}} |x| |y|.$
\item For any $S \in  {\mathcal L}({\mathbb R}^{2n})$,  the random variable $\xi_S$ is subgaussian, and $\| \xi_S \|_{\psi_2} \leq {\frac {c} {\sqrt n}}  \|S\|_1$.
\end{enumerate}%
Here $\| \cdot \|_{\psi_2} $ is the Orlicz norm introduced in Section~\ref{sub-sec-orlicz} above.
\end{lemma}

\begin{proof}[{\bf Proof of Lemma~\ref{lemma-about-Schatten}}]
Let $(x,y) \in {\mathbb R}^{2n} \times {\mathbb R}^{2n}$. Note that we can assume that $|x||y| \neq 0$, and that $x$ and $y$ are not collinear. Denote $e = {\frac x {|x|}}$ and $f = {\frac {Py} {|Py|} }$, where $P$ is the orthogonal projection on $\{x\}^{\perp}$, i.e., $Py = y - \langle y,e \rangle e$. Since $J(O)x \perp x$, one has that
$$ \xi_{x,y} = |x|  |Py| \, \langle J(O)e,f\rangle.$$
From Corollary~\ref{cor-about-distribution1} it follows that for a random $O \in {\rm SO}(2n)$ distributed according to the Haar measure $\mu$, the vector $J(O)e$ is uniformly distributed on $S^{2n-2} \cong S^{2n-1} \cap \{e\}^{\perp}$ with respect to the measure $\sigma_{2n-2}$ on $S^{2n-2}$. This means that
$(|x|  |Py|)^{-1}\xi_{x,y}$ distributed as the random variable $\zeta_1$ defined on $S^{2n-2}$ by the projection map $S^{2n-2} \ni (\zeta_1,\ldots,\zeta_{2n-1}) \mapsto \zeta_1$. It is well known (see e.g.,~\cite{FLM}) that such a spherical random vector is subgaussian, and that $\|\zeta_1\|_{\psi_2} \asymp {\frac 1 {\sqrt n}}$ (see also the remark at the end of Section~\ref{sub-sec-orlicz}).
Thus we conclude that \begin{equation*} \label{est-for-xi} 
 \|\xi_{x,y}\|_{\psi_2}  \leq {\frac {c} {\sqrt n}} |x| |y|, \end{equation*}
for some universal constant $c>0$. This completes the proof of the first part of the lemma.

Next, by the singular value decomposition theorem (see e.g., Theorem 4.1 in~\cite{GGK}), there exists two orthonormal basis $\{e_k\}$ and $\{f_k\}$ of ${\mathbb R}^{2n}$ such that for every $x \in {\mathbb R}^{2n}$ one has $$Sx =\sum_{i=1}^{2n} s_k \langle x , e_k \rangle f_k,$$
where $\{s_k \}$ are the singular values of $S$. This implies that
\begin{equation*} \label{est-for-xiS} \xi_S(O) = {\rm Tr} (J(O)S) = \sum_{k=1}^{2n} s_k \langle J(O)f_k,e_k \rangle = \sum_{k=1}^{2n} s_k \xi_{f_k,e_k}.\end{equation*}
The proof of the second part of Lemma~\ref{lemma-about-Schatten} now follows from the triangle inequality and the first part of the lemma.
\end{proof}

Next, let $G$ be a $2n \times 2n$ matrix whose entries are  standard i.i.d. ${\mathcal N}(0,1)$ gaussian random variables.
For a linear operator $S \in {\mathcal L}({\mathbb R}^{2n})$ and a pair $(x, y) \in {\mathbb R}^{2n} \times   {\mathbb R}^{2n}$,
we define  two random variables 
analogously to~$(\ref{def-of-xi-s})$ and~$(\ref{def-of-xi-x-y})$ via:
$$ \eta_S(G) = {\frac 1 {\sqrt{2n}}}  {\rm Tr}(GS), \ {\rm and} \ \,  \eta_{x,y}(G) = {\frac 1 {\sqrt{2n}}} \langle Gx,y \rangle  .$$
Clearly, $\eta_S$ and $ \eta_{x,y}$ are centered gaussian random variables, and 
$$  \| \eta_S \|_2 := \left ({\mathbb E} \|\eta_S\|_2^2 \right )^{1/2} = {\frac 1 {\sqrt {2n}}} \|S\|_2.$$
Moreover, it is well known (and can be easily checked) that $L_{\psi_2} \subseteq L_2$, and moreover that 
$  \| \eta_S \|_2  \leq \| \eta_S \|_{\psi_2}$.
Hence, one has \begin{equation} \label{norm-psi2-eta-S}  {\frac {1} {\sqrt {2n}}} \|S\|_2 \leq \| \eta_S \|_{\psi_2}.\end{equation}

\begin{proposition} \label{prop-cor-from-Talagrand}
Let ${\mathcal T} \subset {\mathbb R}^{2n} \times {\mathbb R}^{2n}$ be a compact set, and $\xi_{x,y}, \eta_{x,y}$ as above. Then,
$$ {\mathbb E} \sup_{(x,y) \in {\mathcal T}} \xi_{x,y} \leq C \, {\mathbb E}  \sup_{(x,y) \in {\mathcal T}} \eta_{x,y},$$
where $C>0$ is some universal constant.
\end{proposition}

\begin{proof}[{\bf Proof of Proposition~\ref{prop-cor-from-Talagrand}}]
Let $t_1 = (x_1,y_1), t_2 = (x_2,y_2)$ be two points in ${\mathbb R}^{2n} \times {\mathbb R}^{2n}$.
Denote $S:= x_1 \otimes y_1-x_2 \otimes y_2$, where for vectors $v,u \in {\mathbb R}^n$.
 Note that, by definition, $\xi_S = \xi_{x_1,y_1} - \xi_{x_2,y_2}$ and $\eta_S = \eta_{x_1,y_1} - \eta_{x_2,y_2}$. Moreover, from Lemma~\ref{lemma-about-Schatten} it follows that
\begin{equation} \label{norm-psi2-xi-S}  \| \xi_S \|_{\psi_2} = \| \xi_{x_1,y_1} - \xi_{x_2,y_2}\|_{\psi_2} \leq {\frac {c} {\sqrt {n}}} \|S\|_1, \end{equation}
where $c >0$ is the constant appearing in Lemma~\ref{lemma-about-Schatten}. On the other hand, since by definition ${\rm rank}(S) \leq 2$, one has $\|S\|_1 \leq \sqrt{2} \, \|S\|_2$. Thus, from~$(\ref{norm-psi2-eta-S})$ and~$(\ref{norm-psi2-xi-S})$ we conclude that
$$  \| \xi_{x_1,y_1} - \xi_{x_2,y_2}\|_{\psi_2}   \leq
2c  \| \eta_{x_1,y_1} - \eta_{x_2,y_2} \|_{\psi_2}. $$
The proof 
now follows from Talagrand's comparison result (Theorem~\ref{Talagrand-Comparison-Theorem} above).
\end{proof}
\begin{cor} \label{first-proposition}
Let $K_1,K_2 \subset {\mathbb R}^{2n}$ be two centrally-symmetric convex bodies, $O \in  {\rm SO}(2n)$, and $\widetilde G = {\frac 1 {\sqrt{2n}}} G$, where $G$ is a $2n \times 2n$ matrix whose entries are  standard i.i.d. ${\mathcal N}(0,1)$ gaussian random variables.
Then, one has \begin{equation*} \label{Coro-from-Tala} {\mathbb E}_{\mu} \| J(O)  \|_{K_1 \rightarrow K_2} \leq C \, {\mathbb E} \| \widetilde G \|_{K_1 \rightarrow K_2}, \end{equation*}
for some universal constant $C>0$.
\end{cor}

\begin{proof}[{\bf Proof of Corollary~\ref{first-proposition}}]
The proof follows immediately from the fact that for every linear operator $A \in {\mathcal L}({\mathbb R}^{2n})$ one has $\| A \|_{K_1 \rightarrow K_2} = \sup_{(x,y)\in K_1 \times K_2^{\circ}} \langle Ax,y \rangle$, combined with Proposition~\ref{prop-cor-from-Talagrand}, when one takes the set ${\mathcal T}$ to be ${\mathcal T} = K_1 \times K_2^{\circ}$.
\end{proof}



We are now in a position to prove our main results. 

\begin{proof}[{\bf Proof of Theorem~\ref{main-theorem}}]
%
Note first that for every $O \in {\rm SO}(2n)$ one has
\begin{equation} \label{upper-bound-for-norm-J} \alpha(OK) 
= \sup_{x \in K^{\circ}} \| J(O)x||_{K} \geq \| J(O) v||_{K} \geq \sup_{w \in K^{\circ} \cap L}  \langle J(O)v,w \rangle.\end{equation}
Next, it follows from Corollary~\ref{cor-about-distribution1} that for a random (with respect to the Haar measure $\mu$) rotation $O \in {\rm SO}(2n)$,
the vector  $J(O)v$ is uniformly
distributed over the $(2n-2)$-dimensional sphere $S^{2n-2}(r) = S^{2n-1}(r) \cap L$, with radius $r=|v|=R(K^{\circ})$.  Thus, after re-scaling, we obtain that
\begin{equation} \label{one-side-of-the-esitmate-of-the-expectaiton-aa} {\mathbb E}_{\mu}  \left (  \sup_{w \in K^{\circ} \cap L}  \langle J(O)v,w \rangle \right)  = R(K^{\circ})
\int\limits_{x \in S^{2n-2}}  \sup_{w \in K^{\circ} \cap L}  \langle x,w \rangle d \sigma_{2n-2}
 = R(K^{\circ})  M^*(K^{\circ} \cap L). \end{equation}
Combining this with~$(\ref{upper-bound-for-norm-J})$ 
we conclude that
 \begin{equation} \label{one-side-of-the-esitmate-of-the-expectaiton} R(K^{\circ})  M^*(K^{\circ} \cap L) \leq {\mathbb E}_{\mu}  \left (  \alpha(OK) \right). \end{equation}
To get an upper bound for the expectation ${\mathbb E}_{\mu}   \left (  \alpha(OK) \right)$, we consider the symmetric convex body $K^{\circ}_L := K^{\circ} \cap L$. 
We denote by $P_L$  the orthogonal projection to the subspace $L =\{v\}^{\perp}$, and by $P_v$ the orthogonal projection to ${\rm Span}\{v\}$.
Note that  $P_v x + P_L x = x$ for every $x \in {\mathbb R}^{2n}$.
From the fact that the vector $v$ is one of the contact points between the body $K^{\circ}$ and its minimal circumscribed ball it follows that for every $x \in {\mathbb R}^{2n}$ one has $\| P_v x \|_{K^{\circ}} \leq \|x\|_{K^{\circ}}$, and hence also $\|P_Lx\|_{K^{\circ}} \leq 2 \|x\|_{K^{\circ}} $. Thus, for every $x \in {\mathbb R}^{2n}$ 
$$\|x\|_{K^{\circ}} \leq \| P_v x \|_{K^{\circ}} + \|P_Lx\|_{K^{\circ}}   \leq 3\|x\|_{K^{\circ}}.$$ Geometrically, this means that
\begin{equation*} \label{subset-with-3*} {\rm Conv}\{ \pm v, K^{\circ}_L \} \subseteq K^{\circ} \subseteq 3{\rm Conv}\{ \pm v, K^{\circ}_L \}. \end{equation*}
From this it follows that
\begin{equation} \label{technical-point1}
  \alpha(OK)  
\leq 9 \Biggl(  \sup_{w \in K^{\circ} \cap L}  \langle J(O)v,w \rangle +  \sup_{u,w \in K^{\circ} \cap L}  \langle J(O)u,w \rangle \Biggr).
\end{equation}
Note that the expectation with respect to the Haar measure $\mu$ of the first term on the right-hand side of~$(\ref{technical-point1})$ is given by~$(\ref{one-side-of-the-esitmate-of-the-expectaiton-aa})$ above.
To estimate the expectation of the second term 
we combine Corollary~\ref{first-proposition}  
with Chevet's inequality (Theorem~\ref{Chevet's-Inequality}) to conclude that
\begin{equation} \label{other-side-of-the-esitmate-of-the-expectaiton}
 {\mathbb E}_{\mu} \left ( \sup_{u,w \in K^{\circ} \cap L}  \langle J(O)u,w \rangle \right) 
 \leq C'  \bigl (   R(K^{\circ}) M^*(K^{\circ} \cap L) \bigr ),
\end{equation}
for some universal constant $C' >0$.
Hence, from~$(\ref{one-side-of-the-esitmate-of-the-expectaiton-aa})$,~$(\ref{technical-point1})$, and~$(\ref{other-side-of-the-esitmate-of-the-expectaiton})$ it follows that
\begin{equation} \label{finally-the-upper-bound}   {\mathbb E}_{\mu} \left (  \alpha(OK) \right) \leq C_1 R(K^{\circ}) M^*(K^{\circ} \cap L),\end{equation}
for some other universal constant $C_1 >0$. The combination of~$(\ref{one-side-of-the-esitmate-of-the-expectaiton})$ and~$(\ref{finally-the-upper-bound})$ completes the proof of the first part of Theorem~\ref{main-theorem}.

To prove the second part of the theorem, we shall use Gromov-Milman concentration inequality (Theorem~\ref{concentration-orthogonal} above), and an estimation of the Lipschitz constant of the function $O \mapsto \alpha(OK)$ defined on ${\rm SO}(2n)$. For this end,
 note that $$ \alpha(OK)  
= \sup_{A \in \mathcal S} {\rm Tr}(J(O)A),$$
where the supremum is taken over all element  in the set $ {\mathcal S} = K^{\circ} \otimes K^{\circ} = \{x \otimes y \, : \, x,y \in K^{\circ} \}$.
Thus, for every $O_1,O_2 \in {\rm SO}(2n)$, one has
\begin{equation} \label{tri-ineq1}  \alpha(O_1K) - \alpha(O_2K) \leq \|   J(O_1) -  J(O_2)  \|_{K^{\circ} \to K} =  \sup_{A \in \mathcal S} {\rm Tr} \bigl((J(O_1)-J(O_2))A \bigr). \end{equation}
Using the fact that for two square matrices one has ${\rm Tr}(AB) \leq \|A\|_2\|B\|_2$, and $\|AB\|_2 \leq \|A\|_2 \|B\|$ (where $\|B\|$ is the operator norm), 
it follows that for a fixed $A \in {\mathcal S} $, 
\begin{equation*}
  \label{eq:estimate-HS-norm1}
  \begin{aligned}
   {\rm Tr}((J(O_1)-J(O_2))A)  
           &=  {\rm Tr}(J(O_1)A) -   {\rm Tr}(O_1^TJO_2A) + {\rm Tr}(O_1^TJO_2A)  - {\rm Tr}(J(O_2)A)  \\
         &= {\rm Tr}(O_1^TJ( O_1-O_2)A)) + {\rm Tr}((O_1^T-O_2^T)J O_2A))  \\         &\leq  \| O_1^TJ( O_1-O_2) \|_2 \|A \|_2 +   \| (O_1^T-O_2^T)JO_2 \|_2 \|A \|_2
         \\ & = 2 \|A\|_2 \|O_1 - O_2\|_2.
  \end{aligned}
\end{equation*}
Using this estimate, we conclude from~$(\ref{tri-ineq1})$ that
$$  \alpha(O_1K) - \alpha(O_2K) \leq 2 \sup_{A \in \mathcal S} \|A\|_2 \|O_1 - O_2\|_2.$$
On the other hand, from the definition of the set ${\mathcal S}$ it follows that
\begin{equation*} \label{est-otimes} \sup\limits_{A \in {\mathcal S}} ||A ||_2 = \sup\limits_{x,y \in K^{\circ}} |x| |y| = R(K^{\circ})^2. \end{equation*}
Combining the above two inequalities  
we conclude that 
\begin{equation} \label{est-Lip-first-final}
 \alpha(O_1K)-\alpha(O_2K) \leq 2 R(K^{\circ})^2 \|O_1 -O_2\|_2.
\end{equation}
The concentration inequality in Theorem~\ref{main-theorem} now follows from estimate~$(\ref{est-Lip-first-final}$) and
Theorem~\ref{concentration-orthogonal} above. This completes the proof of the theorem.
\end{proof}

\begin{proof}[{\bf Proof of Theorem~\ref{our-main-theorem-second-version}}]
From the assumption that $K^{\circ} \cap L$ is $(C,q)$-non-degenerate it follows that
\begin{equation} \label{cor-from-non-degenerate26}
\int\limits_{x \in S^{2n-2}(r)} \sup_{w \in K^{\circ} \cap L} \langle x,w \rangle \sigma_{2n-2} \leq C \Bigl ( \int\limits_{x \in S^{2n-2}(r)} \bigl ( \sup_{w \in K^{\circ} \cap L} \langle x,w \rangle  \bigr)^{-q} \sigma_{2n-2}  \Bigr )^{-1/q},
\end{equation}
where $S^{2n-2}(r) = S^{2n-1}(r) \cap L$ is a $(2n-2)$-dimensional sphere of radius $r =|v| = R(K^{\circ})$.
From~$(\ref{upper-bound-for-norm-J})$ and 
Corollary~\ref{cor-about-distribution1} we conclude that for every $p>0$
\begin{equation*} \label{expectations-capacity-new11}  
{\mathbb E}_{\mu} \Bigl (   \bigl (  \alpha(OK)   \bigr)^{-p} \Bigr ) \leq 
 \int\limits_{x \in S^{2n-2}(r)}  \left ( \sup\limits_{w \in K^{\circ} \cap L} \langle x,w \rangle \right)^{-p} d\sigma_{2n-2}(x). \end{equation*}
This together with~$(\ref{cor-from-non-degenerate26}$), H\"older's inequality, and~$(\ref{one-side-of-the-esitmate-of-the-expectaiton}$) gives that for every $0 < p \leq q$
\begin{equation} \label{lower-bound-non-deg4}
\Bigl ({\mathbb E}_{\mu} \Bigl (   \bigl (  \alpha(OK)   \bigr)^{-p} \Bigr ) \Bigr)^{1/p} \leq \Bigl ({\mathbb E}_{\mu} \Bigl (   \bigl (  \alpha(OK)   \bigr)^{-q} \Bigr ) \Bigr)^{1/q} \leq
 {\frac {C} {R(K^{\circ})M^*(K^{\circ} \cap L)}}.
 \end{equation}

On the other hand, for any strictly positive random variable $X$ and any $p>0$, one has
${\mathbb E}(X^{-p}) \geq ({\mathbb E}(X))^{-p}$ (e.g., via Jensen's inequality). This together with~$(\ref{finally-the-upper-bound})$ above
 immediately imply   that
\begin{equation} \label{expectations-capacity-new14} 
\Bigl ({\mathbb E}_{\mu} \Bigl (   \bigl (  \alpha(OK)   \bigr)^{-p} \Bigr ) \Bigr)^{1/p}  \geq {\frac {C_0} {R(K^{\circ})M^*(K^{\circ} \cap L)} }, \end{equation}
where $C_0 = (C_1)^{-1}$, and $C_1$ is the constant appearing in inequality~$(\ref{finally-the-upper-bound})$ above.
The combination of~$(\ref{lower-bound-non-deg4})$,~$(\ref{expectations-capacity-new14})$, Theorem~\ref{GO-thm}, and the fact that for a centrally symmetric convex body $r(K) = R(K^{\circ})^{-1}$ completes the proof of the Theorem~\ref{our-main-theorem-second-version}.
\end{proof}

\begin{proof}[{\bf Proof of Corollary~\ref{cor-case-p=1}}]
The second part of Corollary~\ref{cor-case-p=1} follows immediately from Theorem~\ref{our-main-theorem-second-version}. 
For the concentration estimate, we use again Gromov-Milman concentration inequality (Theorem~\ref{concentration-orthogonal}), this time combined with an estimate of the Lipschitz constant of the map
 $\zeta_K : {\rm SO}(2n) \rightarrow {\mathbb R}$ given by
$$ \zeta_K(O) =   \left( \alpha(OK) \right)^{-1}.$$
Note that from the definition of $\alpha$, it follows that for every $O \in {\rm SO}(2n)$, one has the lower bound: $\alpha(OK) \geq r^2(K^{\circ}) = R^{-2}(K)$. Combining this with estimate~$(\ref{est-Lip-first-final})$ we conclude that for any $O_1,O_2 \in {\rm SO}(2n)$ one has
\begin{equation} \label{lip-est-2}
\zeta_K(O_1)-\zeta_K(O_2) =   {\frac {\alpha(O_2K)-\alpha(O_1K)} {\alpha(O_1K)\alpha(O_2K) } } \leq  {2R^2(K^{\circ})} {R^2(K)}  \|O_1-O_2\|_2. \end{equation}
The proof of the corollary now follows from Theorem~\ref{concentration-orthogonal}.
\end{proof}

\smallskip \noindent
Efim ben David Gluskin \\
School of Mathematical Sciences \\
Tel Aviv University, Tel Aviv 69978, Israel  \\
{\it e-mail}: gluskin@post.tau.ac.il \\

\smallskip \noindent
Yaron Ostrover \\
School of Mathematical Sciences \\
Tel Aviv University, Tel Aviv 69978, Israel  \\
{\it e-mail}: ostrover@post.tau.ac.il \\

\end{document}